\documentclass[12pt]{amsart}

\usepackage{amssymb}  
\usepackage{latexsym} 

\usepackage[all]{xy}

\newcommand{\Z}{{\mathbb Z}}
\newcommand{\Q}{{\mathbb Q}}
\newcommand{\R}{{\mathbb R}}
\newcommand{\F}{{\mathbb F}}
\newcommand{\BP}{{\mathbb P}}
\newcommand{\To}{\longrightarrow}
\newcommand{\Pic}{\operatorname{Pic}}
\newcommand{\Div}{\operatorname{Div}}
\newcommand{\inj}{\hookrightarrow}
\newcommand{\diw}{\operatorname{div}}
\newcommand{\CO}{{\mathcal O}}

\newtheorem{Theorem}{Theorem}
\newtheorem{Lemma}[Theorem]{Lemma}
\newtheorem{Proposition}[Theorem]{Proposition}

\theoremstyle{definition}
\newtheorem{Remark}[Theorem]{Remark}

\numberwithin{equation}{section}

\begin{document}

\title{On a Problem of Hajdu and Tengely}

\author{Samir Siksek}
\address{Institute of Mathematics,
         University of Warwick,
         Coventry CV4 7AL, United Kingdom}
\email{s.siksek@warwick.ac.uk}

\author{Michael Stoll}
\address{Mathematisches Institut,
         Universit\"at Bayreuth,
         95440 Bayreuth, Germany.}
\email{Michael.Stoll@uni-bayreuth.de}


\date{29 May, 2010}

\begin{abstract}
  We prove a result that finishes the study of primitive arithmetic
  progressions consisting of squares and fifth powers that was
  carried out by Hajdu and Tengely in a recent paper:
  The only arithmetic progression in coprime integers of the form
  $(a^2, b^2, c^2, d^5)$ is $(1, 1, 1, 1)$.
  For the proof, we first reduce the problem to that of determining
  the sets of rational points on three specific hyperelliptic curves
  of genus~4. A 2-cover descent computation shows that there are no
  rational points on two of these curves. We find generators for a
  subgroup of finite index of the Mordell-Weil group of the last curve.
  Applying Chabauty's method, we prove that
  the only rational points on this curve are the obvious ones.
\end{abstract}

\maketitle


\section{Introduction}

Euler (\cite[pages 440 and 635]{Dickson})
proved Fermat's claim that four distinct squares cannot
form an arithmetic progression. Powers in arithmetic progressions are
still a subject of current interest.
For example, Darmon and Merel \cite{DM} proved that the only solutions
in coprime integers to the Diophantine equation $x^n+y^n=2z^n$
with $n \geq 3$ satisfy $xyz=0$ or $\pm 1$. This shows that
there are no non-trivial three term arithmetic progressions
consisting of $n$-th powers with $n \geq 3$.
The result of Darmon and Merel is far from elementary; it needs
all the tools used in Wiles' proof of Fermat's Last Theorem
and more.

An arithmetic
progression $(x_1,x_2,\ldots,x_k)$ of integers is said to be {\em primitive}
if the terms are coprime, i.e., if $\gcd(x_1,x_2)=1$. Let $S$ be a finite subset of integers
$\geq 2$. Hajdu \cite{Hajdu} showed that if
\begin{equation}\label{eqn:ap}
(a_1^{\ell_1},\ldots, a_k^{\ell_k})
\end{equation}
is a non-constant primitive arithmetic progression with $\ell_i \in S$,
then $k$ is bounded by some (inexplicit) constant $C(S)$.
Bruin, Gy\H{o}ry, Hajdu and Tengely \cite{BGHT} showed
that for any $k \geq 4$ and any $S$, there are only finitely
many primitive arithmetic progressions of the form \eqref{eqn:ap},
with $\ell_i \in S$. Moreover, for $S=\{2,3\}$ and $k \geq 4$,
they showed that $a_i = \pm 1$ for $i=1,\ldots,k$.

A recent paper of Hajdu and Tengely \cite{HT} studies primitive
arithmetic progressions \eqref{eqn:ap} with exponents
belonging to $S=\{2,n\}$
and $\{3,n\}$. In particular, they show that any primitive non-constant
arithmetic progression \eqref{eqn:ap} with exponents $\ell_i \in \{2,5\}$
has $k \leq 4$. Moreover, for $k=4$ they show that
\begin{equation}\label{eqn:l}
(\ell_1,\ell_2,\ell_3,\ell_4) =(2,2,2,5) \quad \text{or}
\quad (5,2,2,2).
\end{equation}
Note that if $(a_i^{\ell_i} : i=1,\ldots,k)$ is
an arithmetic progression, then so is the reverse
progression $(a_i^{\ell_i}: i=k,k-1,\ldots,1)$.
Thus there is really only one case left open by Hajdu and Tengely,
with exponents $(\ell_1,\ell_2,\ell_3,\ell_4) =(2,2,2,5)$.
This is also mentioned as Problem~11
in a list of 22 open problems recently compiled by
Evertse and Tijdeman~\cite{LeidenProblem}.
In this paper we deal with this case.
\begin{Theorem} \label{Thm}
  The only arithmetic progression in coprime integers of the form
  \[ (a^2, b^2, c^2, d^5) \]
  is $(1, 1, 1, 1)$.
\end{Theorem}

 This together with the above-mentioned results of Hajdu
and Tengely completes
the proof of the following theorem.

\begin{Theorem}
There are no non-constant primitive arithmetic progressions
of the form \eqref{eqn:ap}
 with $\ell_i \in \{2,5\}$ and $k \geq 4$.
\end{Theorem}

The primitivity condition is crucial, since otherwise solutions
abound. Let for example $(a^2, b^2, c^2, d)$ be any arithmetic
progression whose first three terms are squares --- there are infinitely
many of these; one can take
$a = r^2 - 2rs - s^2$, $b = r^2 + s^2$, $c = r^2 + 2rs - s^2 $ ---
then $\bigl((ad^2)^2, (bd^2)^2, (cd^2)^2, d^5)$ is an arithmetic
progression
whose first three terms are squares and whose last term is a fifth power.

For the proof of Thm.~\ref{Thm},
   we first reduce the problem to that of determining
  the sets of rational points on three specific hyperelliptic curves
  of genus~4. A $2$-cover descent computation
        (following Bruin and Stoll \cite{BSTwoCoverDesc})
  shows that there are no
  rational points on two of these curves. We find generators for a
  subgroup of finite index of the Mordell-Weil group of the last curve.
  Applying Chabauty's method, we prove that
  the only rational points on this curve are the obvious ones.
All our computations are performed using the computer package
{\sf MAGMA}~\cite{MAGMA}.

The result we prove here may perhaps not be of compelling interest
in itself. Rather, the purpose of this paper is to demonstrate
how we can solve problems of this kind with the available machinery.
We review the relevant part of this machinery in Sect.~\ref{S:Back},
after we have constructed the curves pertaining to our problem in Sect.~\ref{Curves}.
Then, in Sect.~\ref{S:Points}, we apply the machinery to these curves.
The proofs are mostly computational. We have tried to make it clear
what steps need to be done, and to give enough information to make it
possible to reproduce the computations (which have been performed
independently by both authors as a consistency check).


\section{Construction of the Curves} \label{Curves}

Let $(a^2, b^2, c^2, d^5)$ be an arithmetic progression in coprime integers.
Since a square is $\equiv 0$ or $1 \bmod 4$, it follows that all terms are
$\equiv 1 \bmod 4$, in particular, $a$, $b$, $c$ and $d$ are all odd.

Considering the last three terms, we have the relation
\[ (-d)^5 = b^2 - 2 c^2 = (b + c \sqrt{2}) (b - c \sqrt{2}) \,. \]
Since $b$ and~$c$ are odd and coprime,
the two factors on the right are coprime in
$R = \Z[\sqrt{2}]$. Since $R^\times/(R^\times)^5$ is generated
by $1 + \sqrt{2}$, it follows that
\begin{equation} \label{rel1}
  b + c \sqrt{2} = (1 + \sqrt{2})^j (u + v \sqrt{2})^5
                 = g_j(u,v) + h_j(u,v) \sqrt{2}
\end{equation}
with $-2 \le j \le 2$ and $u, v \in \Z$ coprime (with $u$ odd and
$v \equiv j+1 \bmod 2$).
The polynomials $g_j$ and $h_j$ are homogeneous of degree~5 and have
coefficients in~$\Z$.

Now the first three terms of the progression give the relation
\[ a^2 = 2 b^2 - c^2 = 2 g_j(u,v)^2 - h_j(u,v)^2 \,. \]
Writing $y = a/v^5$ and $x = u/v$, this gives the equation of a hyperelliptic
curve of genus~4,
\[ C_j : y^2 = f_j(x) \]
where $f_j(x) = 2 g_j(x,1)^2 - h_j(x,1)^2$. Every arithmetic progression of
the required form therefore induces a rational point on one of the curves~$C_j$.

We observe that taking conjugates in~\eqref{rel1} leads to
\[ (-1)^j b + (-1)^{j+1} c\sqrt{2}
    = (1 + \sqrt{2})^{-j} (u + (-v) \sqrt{2})^5 \,,
\]
which implies that $f_{-j}(x) = f_j(-x)$ and therefore that $C_{-j}$ and~$C_j$
are isomorphic and their rational points correspond to the same arithmetic
progressions. We can therefore restrict attention to $C_0$, $C_1$ and $C_2$.
Their equations are as follows.
\begin{align*}
  C_0 : y^2 &= f_0(x) = 2 x^{10} + 55 x^8 + 680 x^6 + 1160 x^4 + 640 x^2 - 16 \\
  C_1 : y^2 &= f_1(x) = x^{10} + 30 x^9 + 215 x^8 + 720 x^7 + 1840 x^6 + 3024 x^5 \\
            & \qquad\qquad\qquad + 3880 x^4 + 2880 x^3 + 1520 x^2 + 480 x + 112 \\
  C_2 : y^2 &= f_2(x)
             = 14 x^{10} + 180 x^9 + 1135 x^8 + 4320 x^7 + 10760 x^6 + 18144 x^5 \\
            & \qquad\qquad\qquad + 21320 x^4 + 17280 x^3 + 9280 x^2 + 2880 x + 368
\end{align*}

The trivial solution $a = b = c = d = 1$ corresponds to $j = 1$, $(u,v) = (1,0)$
in the above and therefore gives rise to the point $\infty_+$ on~$C_1$
(this is the point at infinity where $y/x^5$ takes the value~$+1$). Changing
the signs of $a$, $b$ or $c$ leads to $\infty_- \in C_1(\Q)$ (the point where
$y/x^5 = -1$) or to the two points at infinity on the isomorphic curve~$C_{-1}$.


\section{Background on Rational Points on Hyperelliptic Curves}
\label{S:Back}

Our task will be to determine the set of rational points on each of the
curves $C_0$, $C_1$ and~$C_2$ constructed in the previous section. In this
section, we will give an overview of the methods we will use, and in the
next section, we will apply these methods to the given curves.

We will restrict attention to {\em hyperelliptic} curves, i.e., curves
given by an affine equation of the form
\[ C : y^2 = f(x) \]
where $f$ is a squarefree polynomial with integral coefficients. The smooth
projective curve birational to this affine curve has either one or two
additional points `at infinity'. If the degree of~$f$ is odd, there is one point
at infinity, which is always a rational point. Otherwise there are
two points at infinity corresponding to the two square roots of the leading
coefficient of~$f$. In particular, these two points are rational if and only
if the leading coefficient is a square. For example, $C_1$ above has two
rational points at infinity, whereas the points at infinity on $C_0$ and~$C_2$
are not rational. We will use $C$ in the following to denote the smooth
projective model; $C(\Q)$ denotes as usual the set of rational points
including those at infinity.


\subsection{Two-Cover Descent}
\label{SS:Twocov}

It will turn out that $C_0$ and~$C_2$ do not have rational points. One way
of showing that $C(\Q)$ is empty is to verify that $C(\R)$ is empty or that
$C(\Q_p)$ is empty for some prime~$p$. This does not work for $C_0$ or~$C_2$;
both curves have real points and $p$-adic points for all~$p$. (This can be
checked by a finite computation.) So we need a more sophisticated way of showing
that there are no rational points. One such method is known as {\em 2-cover descent}.
We sketch the method here; for a detailed description, see~\cite{BSTwoCoverDesc}.

An important ingredient of this and other methods is the algebra
\[ L := \Q[T] = \frac{\Q[x]}{\Q[x] \cdot f(x)} \,, \]
where $T$ denotes the image of~$x$. If $f$ is irreducible (as in our examples),
then $L$ is the number field generated by a root of~$f$. In general, $L$ will
be a product of number fields corresponding to the irreducible factors of~$f$.
We now assume that $f$ has even degree $2g+2$, where $g$ is the genus of the
curve. This is the generic case; the odd degree case is somewhat simpler.
We can then set up a map, called the {\em descent map} or {\em $x-T$ map}:
\[ x-T : C(\Q) \To H := \frac{L^\times}{\Q^\times (L^\times)^2} \,. \]
Here $L^\times$ denotes the multiplicative group of~$L$, and $(L^\times)^2$
denotes the subgroup of squares. On points $P \in C(\Q)$ that are neither
at infinity nor Weierstrass points (i.e., points with vanishing $y$ coordinate),
the map is defined as
\[ (x-T)(P) = x(P) - T \bmod \Q^\times (L^\times)^2 \,. \]
Rational points at infinity map to the trivial element, and if there are
rational Weierstrass points, their images can be determined using the fact
that the norm of $x(P) - T$ is $y(P)^2$ divided by the leading coefficient
of~$f$. If we can show that $x-T$ has empty image on~$C(\Q)$, then it
follows that $C(\Q)$ is empty.

We obtain information of the image by considering again $C(\R)$ and~$C(\Q_p)$.
We can carry out the same construction over $\R$ and over~$\Q_p$, leading
to an algebra $L_v$ ($v = p$, or $v = \infty$ when working over~$\R$),
a group~$H_v$ and a map
\[ (x-T)_v : C(\Q_v) \To H_v \qquad \text{(where $\Q_\infty = \R$).} \]
We have inclusions $C(\Q) \inj C(\Q_v)$ and canonical homomorphisms
$H \to H_v$. Everything fits together in a commutative diagram
\[ \xymatrix{ C(\Q) \ar[rr]^{x-T} \ar[d] & & H \ar[d] \\
              \prod_v C(\Q_v) \ar[rr]^{\prod_v (x-T)_v} & & \prod_v H_v
            }
\]
where $v$ runs through the primes and~$\infty$. If we can show that the
images of the lower horizontal map and of the right vertical map do not
meet, then the image of $x-T$ and therefore also $C(\Q)$ must be empty.
We can verify this by considering a finite subset of `places'~$v$.

In general, we obtain a finite subset of~$H$ that contains the image
of~$x-T$; this finite subset is known as the {\em fake 2-Selmer set}
of~$C/\Q$. It classifies either pairs of (isomorphism classes of)
2-covering curves of~$C$ that have points {\em everywhere locally},
i.e., over~$\R$ and over all~$\Q_p$,
or else it classifies such 2-covering curves, in which case it is
the (true) 2-Selmer set. Whether it classifies
pairs or individual 2-coverings depends on a certain
condition on the polynomial~$f$. This condition is satisfied if either
$f$ has an irreducible factor of odd degree, or if $\deg f \equiv 2 \bmod 4$
and $f$ factors over a quadratic extension $\Q(\sqrt{d})$ as a constant
times the product of two conjugate polynomials. A {\em 2-covering} of~$C$ is
a morphism $\pi : D \to C$ that is unramified and becomes Galois
over a suitable field extension of finite degree,
with Galois group $(\Z/2\Z)^{2g}$.
It is known that every rational point on~$C$ lifts to a rational point
on some 2-covering of~$C$.

The actual computation splits into a global and a local part. The global
computation uses the ideal class group and the unit group of~$L$ (or the
constituent number fields of~$L$) to construct a finite subgroup of~$H$
containing the image of~$x-T$. The local computation determines the
image of $(x-T)_v$ for finitely many places~$v$.


\subsection{The Jacobian}
\label{SS:Jac}

Most other methods make use of another object associated to the curve~$C$:
its {\em Jacobian variety} (or just {\em Jacobian}). This is an abelian variety~$J$
(a higher-dimensional analogue of an elliptic curve) of dimension~$g$, the
genus of~$C$. It reflects a large part of the geometry and arithmetic of~$C$;
its main advantage is that its points form an abelian group, whereas the
set of points on~$C$ does not carry a natural algebraic structure.

For our purposes, we can more or less forget the structure of~$J$ as a
projective variety. Instead we use the description of the points on~$J$
as the elements of the degree zero part of the {\em Picard group} of~$C$.
The Picard group is constructed as a quotient of the group of divisors on~$C$.
A {\em divisor} on~$C$ is an element of the free abelian group $\Div_C$ on the
set~$C(\bar\Q)$ of all algebraic points on~$C$. The absolute Galois group
of~$\Q$ acts on~$\Div_C$; a divisor that is fixed by this action is {\em rational}.
This does not mean that the points occurring in the divisor must be rational;
points with the same multiplicity can be permuted. A nonzero rational function~$h$
on~$C$ with coefficients in~$\bar\Q$ has an associated divisor $\diw(h)$ that records
its zeros and poles (with multiplicities). If $h$ has coefficients in~$\Q$,
then $\diw(h)$ is rational. The homomorphism $\deg : \Div_C \to \Z$
induced by sending each point in~$C(\bar\Q)$ to~$1$ gives the {\em degree}
of a divisor. Divisors of functions have degree zero.

Two divisors $D, D' \in \Div_C$ are {\em linearly equivalent} if their
difference is the divisor of a function. The equivalence classes are the
elements of the {\em Picard group} $\Pic_C$ defined by the following exact
sequence.
\[ 0 \To \bar\Q^\times \To \bar\Q(C)^\times \stackrel{\diw}{\To} \Div_C
     \To \Pic_C \To 0
\]
Since divisors of functions have degree zero, the degree homomorphism
descends to~$\Pic_C$. We denote its kernel by $\Pic^0_C$. It is a fact
that $J(\bar\Q)$ is isomorphic as a group to~$\Pic^0_C$. The rational
points $J(\Q)$ correspond to the elements of~$\Pic^0_C$ left invariant
by the Galois group. In general it is not true that a point in~$J(\Q)$
can be represented by a rational divisor, but this is the case when
$C$ has a rational point, or at least points everywhere locally.
The most important fact about the group $J(\Q)$ is the statement of
the {\em Mordell-Weil Theorem:} $J(\Q)$ is a {\em finitely generated}
abelian group. For this reason, $J(\Q)$ is often called the
{\em Mordell-Weil group} of $J$ or of~$C$.

If $P_0 \in C(\Q)$, then the map $C \ni P \mapsto [P - P_0] \in J$ is
a $\Q$-defined embedding of $C$ into~$J$. We use $[D]$ to denote the
linear equivalence class of the divisor~$D$. The basic idea of the
methods described below is to try to recognise the points of~$C$ embedded
in this way among the rational points on~$J$.

We need a way of representing elements of~$J(\Q)$. Let $P \mapsto P^-$
denote the {\em hyperelliptic involution} on~$C$; this is the morphism
$C \to C$ that changes the sign of the $y$~coordinate. Then it is easy
to see that the divisors $P + P^-$ all belong to the same class
$W \in \Pic_C$. An effective divisor~$D$ (a divisor such that no point occurs
with negative multiplicity) is {\em in general position} if there is
no point~$P$ such that $D - P - P^-$ is still effective. Divisors in
general position not containing points at infinity can be represented
in a convenient way by pairs of polynomials $(a(x), b(x))$. This pair
represents the divisor~$D$ such that its image on the projective line
(under the $x$-coordinate map) is given by the roots of~$a$; the corresponding
points on~$C$ are determined by the relation $y = b(x)$. The polynomials
have to satisfy the relation $f(x) \equiv b(x)^2 \bmod a(x)$. This
is the {\em Mumford representation} of~$D$. The polynomials $a$ and~$b$
can be chosen to have rational coefficients if and only if $D$ is rational.
(The representation can be adapted to allow for points at infinity occurring
in the divisor.)

If the genus $g$ is even, then it is a fact that every point in~$J(\Q)$
has a unique representation of the form $[D] - nW$ where $D$ is a rational
divisor in general position of degree~$2n$ and $n \ge 0$ is minimal.
The Mumford representation of~$D$ is then also called the Mumford representation
of the corresponding point on~$J$. It is fairly easy to add points
on~$J$ using the Mumford representation, see~\cite{Cantor}. This
addition procedure is implemented in~{\sf MAGMA}, for example.

There is a relation between 2-coverings of~$C$ and the Jacobian~$J$.
Assume $C$ is embedded in~$J$ as above. Then if $D$ is any 2-covering
of~$C$ that has a rational point~$P$, $D$ can be realised as the preimage
of~$C$ under a map of the form $Q \mapsto 2Q + Q_0$ on~$J$, where
$Q_0$ is the image of~$P$ on~$C \subset J$. A consequence of this is
that two rational points $P_1, P_2 \in C(\Q)$ lift
to the same 2-covering if and only if $[P_1 - P_2] \in 2 J(\Q)$.


\subsection{The Mordell-Weil Group}
\label{SS:MW}

We will need to know generators of a finite-index subgroup of the
Mordell-Weil group~$J(\Q)$. Since $J(\Q)$ is a finitely generated abelian
group, it will be a direct sum of a finite torsion part and a free abelian
group of rank~$r$; $r$ is called the {\em rank} of~$J(\Q)$. So what we need
is a set of $r$ independent points in~$J(\Q)$.

The torsion subgroup of~$J(\Q)$ is usually easy to determine. The main
tool used here is the fact that the torsion subgroup injects into~$J(\F_p)$
when $p$ is an odd prime not dividing the discriminant of~$f$. If the orders
of the finite groups~$J(\F_p)$ are coprime for suitable primes~$p$, then
this shows that $J(\Q)$ is torsion-free.

We can find points in~$J(\Q)$ by search. This can be done by searching
for rational points on the variety parameterising Mumford representations
of divisors of degree 2, 4, \dots. We can then check if the points found
are independent by again mapping into~$J(\F_p)$ for one or several primes~$p$.

The hard part is to know when we have found enough points. For this we need
an upper bound on the rank~$r$. This can be provided by a {\em 2-descent}
on the Jacobian~$J$. This is described in detail in~\cite{Stoll2Descent}.
The idea is similar to the 2-cover descent on~$C$ described above
in Sect.~\ref{SS:Twocov}. Essentially we extend the $x-T$ map from points
to divisors. It can be shown that the value of $(x-T)(D)$ only depends
on the linear equivalence class of~$D$.
This gives us a homomorphism from~$J(\Q)$ into~$H$, or more
precisely, into the kernel of the norm map
$N_{L/\Q} : H \to \Q^\times/(\Q^\times)^2$. It can be shown that the
kernel of this $x-T$ map on~$J(\Q)$ is either $2 J(\Q)$, or it contains
$2J(\Q)$ as a subgroup of index~2. The former is the case when $f$ satisfies
the same condition as that mentioned in Sect.~\ref{SS:Twocov}.

We can then bound $(x-T)(J(\Q))$ in much the same way as we did when
doing a 2-cover descent on~$C$. The global part of the computation is
identical. The local part is helped by the fact that we now have a
group homomorphism (or a homomorphism of $\F_2$-vector spaces), so we
can use linear algebra. We obtain a bound for the order of $J(\Q)/2J(\Q)$,
from which we can deduce a bound for the rank~$r$. If we are lucky and
found that same number of independent points in~$J(\Q)$, then we know
that these points generate a subgroup of finite index.

The group containing $(x-T)(J(\Q))$ we compute is known as the
{\em fake 2-Selmer group} of~$J$~\cite{PS}. If the polynomial~$f$ satisfies the
relevant condition, then this fake Selmer group is isomorphic to the
true 2-Selmer group of~$J$ (that classifies 2-coverings of~$J$ that
have points everywhere locally).


\subsection{The Chabauty-Coleman Method}
\label{SS:Chab}

If the rank~$r$ is less than the genus~$g$, there is a method available
that allows us to get tight bounds on the number of rational points on~$C$.
This goes back to Chabauty~\cite{Chabauty}, who used it to prove Mordell's
Conjecture in this case. Coleman~\cite{Coleman} refined the method. We
give a sketch here; more details can be found for example in~\cite{StollChabauty}.

Let $p$ be a prime of good reduction for~$C$ (this is the case when $p$
is odd and does not divide the discriminant of~$f$). We use $\Omega_C^1(\Q_p)$
and~$\Omega_J^1(\Q_p)$ to denote the spaces of regular 1-forms on~$C$ and~$J$
that are defined over~$\Q_p$. If $P_0 \in C(\Q)$ and
$\iota : C \to J$, $P \mapsto [P-P_0]$ denotes the corresponding embedding
of $C$ into~$J$, then the induced map $\iota^* : \Omega_J^1(\Q_p) \to \Omega_C^1(\Q_p)$
is an isomorphism that is independent of the choice of basepoint~$P_0$.
Both spaces have dimension~$g$. There is an integration pairing
\[ \Omega_C^1(\Q_p) \times J(\Q_p) \To \Q_p, \quad
   (\iota^* \omega, Q) \longmapsto \int_0^Q \omega = \langle \omega, \log Q \rangle \,.
\]
In the last expression, $\log Q$ denotes the $p$-adic logarithm on~$J(\Q_p)$
with values in the tangent space of~$J(\Q_p)$ at the origin, and $\Omega^1_J(\Q_p)$
is identified with the dual of this tangent space. If $r < g$, then there are
(at least) $g-r$ linearly independent differentials $\omega \in \Omega_C^1(\Q_p)$
that annihilate the Mordell-Weil group~$J(\Q)$. Such a differential can
be scaled so that it reduces to a non-zero differential $\bar\omega$ mod~$p$.
Now the important fact is that if $\bar\omega$ does not vanish at a point
$\bar{P} \in C(\F_p)$, then there is at most one rational point on~$C(\Q)$
whose reduction is~$\bar{P}$. (There are more general bounds valid when
$\bar\omega$ does vanish at~$\bar{P}$, but we do not need them here.)


\section{Determining the Rational Points} \label{S:Points}

In this section, we determine the set of rational points on the three curves
$C_0$, $C_1$ and~$C_2$. To do this, we apply the methods described in
Sect.~\ref{S:Back}.

We first consider $C_0$ and~$C_2$. We apply the
2-cover-descent procedure described in Sect.~\ref{SS:Twocov} to the
two curves and find that in each case, there are no 2-coverings that have points
everywhere locally. For $C_0$, only 2-adic information is needed in addition
to the global computation, for $C_2$, we need 2-adic and 7-adic information.
Note that the number fields generated by roots of $f_0$ or~$f_2$ are sufficiently
small in terms of degree and discriminant that the necessary class and unit
group computations can be done unconditionally. This leads to the following.

\begin{Proposition} \label{Prop02}
  There are no rational points on the curves $C_0$ and~$C_2$.
\end{Proposition}

\begin{proof}
  The 2-cover descent procedure is available in recent releases of {\sf MAGMA}.
  The computations leading to the stated result can be performed by issuing
  the following {\sf MAGMA} commands.
  
\begin{verbatim}
> SetVerbose("Selmer",2);
> TwoCoverDescent(HyperellipticCurve(Polynomial(
    [-16,0,640,0,1160,0,680,0,55,0,2])));
> TwoCoverDescent(HyperellipticCurve(Polynomial(
    [368,2880,9280,17280,21320,18144,10760,4320,1135,180,14])));
\end{verbatim}
  We explain how the results can be checked independently. We give details
  for~$C_0$ first. The procedure for~$C_2$ is similar, so we only explain
  the differences.
  
  The polynomial~$f_0$ is irreducible, and it can be checked that the number
  field generated by one of its roots is isomorphic to $L = \Q(\!\sqrt[10]{288})$.
  Using {\sf MAGMA} or pari/gp, one checks that this field has trivial
  class group. The finite subgroup~$\tilde{H}$ of~$H$ containing the Selmer~set
  is then given
  as $\CO_{L,S}^\times/(\Z_{\{2,3,5\}}^\times (\CO_{L,S}^\times)^2)$, where
  $S$ is the set of primes in~$\CO_L$ above the `bad primes' 2, 3 and~5.
  The set~$S$ contains two primes
  above~2, of degrees 1 and~4, respectively, and one prime above 3 and~5
  each, of degree~2 in both cases. Since $L$ has two real embeddings and
  four pairs of complex embeddings, the unit rank is~5. The rank (or $\F_2$-dimension)
  of~$\tilde{H}$ is then $7$. (Note that 2 is a square in~$L$.) The descent map takes
  its values in the subset of~$\tilde{H}$ consisting of elements whose norm is
  twice a square. This subset is of size~$32$; elements of~$\CO_L$ representing
  it can easily be obtained. Let $\delta$ be such a representative. We let
  $T$ be a root of~$f_0$ in~$L$ and check that the system of equations
  \[ y^2 = f_0(x), \quad x - T = \delta c z^2 \]
  has no solutions with $x,y,c \in \Q_2$, $z \in L \otimes_{\Q} \Q_2$.
  The second equation leads, after expanding $\delta z^2$ as a $\Q$-linear
  combination of $1, T, T^2, \dots, T^9$, to eight homogeneous quadratic
  equations in the ten unknown coefficients of~$z$. Any solution to these
  equations gives a unique~$x$, for which~$f_0(x)$ is a square. The latter
  follows by taking norms on both sides of $x-T = \delta c z^2$. So we
  only have to check the intersection of eight quadrics in~$\BP^9$ for
  existence of $\Q_2$-points. Alternatively, we evaluate the descent map
  on~$C_0(\Q_2)$, to get its image in~$H_2 = L_2^\times/(\Q_2^\times (L_2^\times)^2)$,
  where $L_2 = L \otimes_{\Q} \Q_2$. Then we check that none of the
  representatives~$\delta$ map into this image.

  When dealing with~$C_2$, the field~$L$ is generated by a root of
  $x^{10} - 6 x^5 - 9$. Since the leading coefficient of~$f_2$ is~$14$,
  we have to add (the primes above)~7 to the bad primes. As before, the
  class group is trivial, and we have the same splitting behaviour of
  2, 3 and~5. The prime~7 splits into two primes of degree~1 and two
  primes of degree~4. The group of $S$-units of~$L$ modulo squares has
  now rank~14, the group $\tilde{H}$ has rank~10, and the subset of~$H$ consisting
  of elements whose norm is 14~times a square has 128~elements. These elements
  now have to be tested for compatibility with the 2-adic and the 7-adic
  information, which can be done using either of the two approaches described
  above. The 7-adic check is only necessary for one of the elements; the 127~others
  are already ruled out by the 2-adic check.
\end{proof}

We cannot hope to deal with~$C_1$ in the same easy manner, since $C_1$ has
two rational points at infinity coming from the trivial solutions. We can still
perform a 2-cover-descent computation, though, and find that there is only
one 2-covering of~$C_1$ with points everywhere locally, which is the covering that
lifts the points at infinity. Only 2-adic information is necessary to show
that the fake 2-Selmer set has at most one element, so we can get this result
using the following {\sf MAGMA} command.
\begin{verbatim}
> TwoCoverDescent(HyperellipticCurve(Polynomial(
    [112,480,1520,2880,3880,3024,1840,720,215,30,1]))
      : PrimeCutoff := 2);
\end{verbatim}
(In some versions of {\sf MAGMA} this returns a two-element set. However, as can be checked
by pulling back under the map returned as a second value, these two elements
correspond to the images of $1$ and~$-1$ in $L^\times/(L^\times)^2 \Q^\times$
and therefore both represent the trivial element. The error is caused by
{\sf MAGMA} using $1$ instead of~$-1$ as a `generator' of $\Q^\times/(\Q^\times)^2$.
This bug is corrected in recent releases.)

The computation can be performed in the same way as for $C_0$ and~$C_2$.
The relevant field~$L$ is generated by a root of $x^{10} - 18 x^5 + 9$; it
has class number~1, and the primes 2, 3 and~5 split in the same way as before.
The subset~$H'$ (in fact a subgroup)
of~$\tilde{H}$ consisting of elements with square norm has size~32.
Of these, only the element represented by~1 is compatible with the 2-adic
constraints.

We remark that by the way it is given, the
polynomial~$f_1$ factors over~$\Q(\sqrt{2})$ into two conjugate factors of
degree~5. This implies that the `fake 2-Selmer set' computed by the 2-cover
descent is the true 2-Selmer set, so that there is really only one 2-covering
that corresponds to the only element of the set computed by the procedure.
We state the result as a lemma. We fix $P_0 = \infty_- \in C_1$ as our basepoint
and write $J_1$ for the Jacobian variety of~$C_1$. Then, as described
in Sect.~\ref{SS:Jac},
\[ \iota : C_1 \To J_1\,, \quad P \longmapsto [P - P_0] \]
is an embedding defined over~$\Q$.

\begin{Lemma} \label{Lemma2J}
  Let $P \in C_1(\Q)$. Then the divisor class $[P - P_0]$ is in $2 J_1(\Q)$.
\end{Lemma}

\begin{proof}
  Let $D$ be the unique 2-covering of~$C_1$ (up to isomorphism) that has
  points everywhere locally. The fact that $D$ is unique follows from the
  computation of the 2-Selmer set.
  Any rational point $P \in C_1(\Q)$ lifts to a rational point on some 2-covering
  of~$C_1$. In particular, this 2-covering then has a rational point, so it
  also satisfies the weaker condition that it has points everywhere locally.
  Since $D$ is the only 2-covering of~$C_1$ satisfying this condition,
  $P_0$ and~$P$ must both lift to a rational point on~$D$. This implies
  by the remark at the end of Sect.~\ref{SS:Jac} that $[P-P_0] \in 2 J_1(\Q)$.
\end{proof}

To make use of this information, we need to know $J_1(\Q)$, or at least a
subgroup of finite index. A computer search reveals two points in~$J_1(\Q)$,
which are given in Mumford representation (see Sect.~\ref{SS:Jac}) as follows.
\begin{align*}
  Q_1 &= \bigl(x^4 + 4 x^2 + \tfrac{4}{5},\quad -16 x^3 - \tfrac{96}{5} x\bigr) \\
  Q_2 &= \bigl(x^4 + \tfrac{24}{5} x^3 + \tfrac{36}{5} x^2 + \tfrac{48}{5} x
                + \tfrac{36}{5},\quad
               -\tfrac{1712}{75} x^3 - \tfrac{976}{25} x^2 - \tfrac{1728}{25} x
                - \tfrac{2336}{25}\bigr)
\end{align*}
We note that $2 Q_1 = [\infty_+ - \infty_-]$; this makes Lemma~\ref{Lemma2J}
explicit for the known two points on~$C_1$.

\begin{Lemma} \label{LemmaGroup}
  The Mordell-Weil group $J_1(\Q)$ is torsion-free, and $Q_1$, $Q_2$ are
  linearly independent. In particular, the rank of $J_1(\Q)$ is at least~2.
\end{Lemma}

\begin{proof}
  The only primes of bad reduction for~$C_1$ are 2, 3 and~5. It is known that
  the torsion subgroup of $J_1(\Q)$ injects into $J_1(\F_p)$ when $p$ is an
  odd prime of good reduction. Since $\#J_1(\F_7) = 2400$ and
  $\#J_1(\F_{41}) = 2633441$
  are coprime, there can be no nontrivial torsion in~$J_1(\Q)$.
  
  We check that the image of $\langle Q_1, Q_2 \rangle$ in~$J_1(\F_7)$ is
  not cyclic. This shows that $Q_1$ and~$Q_2$ must be independent.
\end{proof}

The next step is to show that the Mordell-Weil rank is indeed~2. For this,
we compute the 2-Selmer group of~$J_1$ as sketched in Sect.~\ref{SS:MW}
and described in detail in~\cite{Stoll2Descent}.
We give some details of the computation, since it is outside the scope of
the functionality that is currently provided by {\sf MAGMA} (or any other
software package).

We first remind ourselves that $f_1$ factors over~$\Q(\sqrt{2})$. This implies
that the kernel of the $x-T$ map on~$J(\Q)$ is~$2J(\Q)$. Therefore the
`fake 2-Selmer group' that
we compute is in fact the actual 2-Selmer group of~$J_1$.
Since $J_1(\Q)$ is torsion-free, the order of the 2-Selmer group is
an upper bound for~$2^r$, where $r$ is the rank of~$J_1(\Q)$.

The global computation is the same as that we needed to do for the
\hbox{2-cover} descent. In particular, the Selmer group is contained in the
group~$H'$ from above, consisting of the $S$-units of~$L$ with
square norm, modulo squares and modulo $\{2,3,5\}$-units of~$\Q$.
For the local part of the computation,
we have to compute the image of $J_1(\Q_p)$ under the local $x-T$ map
for the primes~$p$ of bad reduction. We check that there is no 2-torsion
in $J_1(\Q_3)$ and~$J_1(\Q_5)$ ($f_1$ remains irreducible both over~$\Q_3$
and over~$\Q_5$). This implies that the targets of
the local maps $(x-T)_3$ and $(x-T)_5$ are trivial, which means that
these two primes need not be considered as bad primes for the descent
computation. The real locus $C_1(\R)$ is connected, which implies that there
is no information coming from the local image at the infinite place.
(Recall that $C_1$ denotes the smooth projective model of the curve.
The real locus of the affine curve $y^2 = f_1(x)$ has two components,
but they are connected to each other through the points at infinity.)
Therefore, we only need to use 2-adic information in the computation.
We set $L_2 = L \otimes_{\Q} \Q_2$ and compute the natural homomorphism 
\[ \mu_2 : H' \To H_2 = \frac{L_2^\times}{\Q_2^\times (L_2^\times)^2} \,. \]
Let $I_2$ be the image of~$J_1(\Q_2)$ in~$H_2$.
Then the 2-Selmer group is $\mu_2^{-1}(I_2)$.

It remains to compute~$I_2$, which is the hardest part of the computation.
The \hbox{2-torsion} subgroup $J_1(\Q_2)[2]$ has order~2 ($f_1$ splits into
factors of degrees 2 and~8 over~$\Q_2$); this implies that
$J_1(\Q_2)/2 J_1(\Q_2)$ has dimension~$g + 1 = 5$ as an \hbox{$\F_2$-vector} space.
This quotient is generated by the images of $Q_1$ and~$Q_2$ and of three
further points of the form $[D_i] - \tfrac{\deg D_i}{2} W$, where $D_i$
is the sum of  points on~$C_1$ whose $x$-coordinates are the roots of
\begin{align*}
  D_1 : & \quad \bigl(x - \tfrac{1}{2}\bigr) \bigl(x - \tfrac{1}{4}\bigr)\,, \\
  D_2 : & \quad x^2 - 2 x + 6\,, \\
  D_3 : & \quad x^4 + 4 x^3 + 12 x^2 + 36\,, \\
\end{align*}
respectively. These points were found by a systematic search, using the
fact that the local map $(x-T)_2$ is injective in our situation. We can
therefore stop the search procedure as soon as we have found points whose
images generate a five-dimensional $\F_2$-vector space.
We thus find $I_2 \subset H_2$ and then can compute the 2-Selmer group.
In our situation, $\mu_2$ is injective, and the intersection of its image
with~$I_2$ is generated by the images of $Q_1$ and~$Q_2$. Therefore,
the $\F_2$-dimension of the 2-Selmer group is~2.

\begin{Lemma} \label{LemmaRank}
  The rank of $J_1(\Q)$ is~2, and $\langle Q_1, Q_2 \rangle \subset J_1(\Q)$
  is a subgroup of finite odd index.
\end{Lemma}

\begin{proof}
  The Selmer group computation shows that the rank is $\le 2$, and
  Lemma~\ref{LemmaGroup} shows that the rank is $\ge 2$. Regarding the second
  statement, it is now clear that we have a subgroup of finite index.
  The observation stated just before the lemma shows that the given subgroup
  surjects onto the 2-Selmer group under the $x-T$ map. Since the kernel
  of the $x-T$ map is~$2 J_1(\Q)$, this implies that the index is odd.
\end{proof}

Now we want to use the Chabauty-Coleman method sketched in Sect.~\ref{SS:Chab}
to show that $\infty_+$ and~$\infty_-$ are the only rational points on~$C_1$.
To keep the computations reasonably simple, we want to work at $p = 7$, which is
the smallest prime of good reduction.

For $p$ a prime of good reduction, we write $\rho_p$ for the two `reduction mod~$p$'
maps $J_1(\Q) \to J_1(\F_p)$ and $C_1(\Q) \to C_1(\F_p)$.

\begin{Lemma} \label{LemmaRed}
  Let $P \in C_1(\Q)$. Then $\rho_7(P) = \rho_7(\infty_+)$ or
  $\rho_7(P) = \rho_7(\infty_-)$.
\end{Lemma}

\begin{proof}
  Let $G = \langle Q_1, Q_2 \rangle$ be the subgroup of $J_1(\Q)$ generated
  by the two points $Q_1$ and~$Q_2$. We find that $\rho_7(G)$ has index~2
  in~$J_1(\F_7) \cong \Z/10\Z \oplus \Z/240\Z$.
  By Lemma~\ref{LemmaRank}, we know that $(J_1(\Q) : G)$ is
  odd, so we can deduce that $\rho_7(G) = \rho_7(J_1(\Q))$. The group
  $J_1(\F_7)$ surjects onto $(\Z/5\Z)^2$. Since $\rho_7(J_1(G))$ has index~2
  in~$J_1(\F_7)$, $\rho_7(G) = \rho_7(J_1(\Q))$ also surjects onto~$(\Z/5\Z)^2$. This
  implies that the index of~$G$ in~$J_1(\Q)$ is not divisible by~5.
  
  We determine the points $P \in C_1(\F_7)$ such that
  $\iota(P) \in \rho_7(2 J_1(\Q)) = 2 \rho_7(G)$. We find the set
  \[ X_7 = \{\rho_7(\infty_+), \rho_7(\infty_-), (-2, 2), (-2, -2)\}\,. \]
  Note that for any $P \in J_1(\Q)$, we must have $\rho_7(P) \in X_7$
  by Lemma~\ref{Lemma2J}.
  
  Now we look at $p = 13$. The image of~$G$
  in~$J_1(\F_{13}) \cong \Z/10\Z \oplus \Z/2850\Z$ has index~5.
  Since we already know that $(J_1(\Q) : G)$ is not a multiple of~5, this
  implies that $\rho_{13}(G) = \rho_{13}(J_1(\Q))$. As above for $p = 7$,
  we compute the set $X_{13} \subset C_1(\F_{13})$ of points mapping into
  $\rho_{13}(2 J_1(\Q))$. We find
  \[ X_{13} = \{\rho_{13}(\infty_+), \rho_{13}(\infty_-)\} \,. \]
  
  Now suppose that there is $P \in C_1(\Q)$ with
  $\rho_7(P) \in \{(-2,2), (-2,-2)\}$. Then $\iota(P)$
  is in one of two specific cosets in
  $J_1(\Q)/\ker \rho_7 \cong G/\ker \rho_7|_G$. On the other hand,
  we have $\rho_{13}(P) = \rho_{13}(\infty_\pm)$, so that $\iota(P)$
  is in one of two specific cosets in
  $J_1(\Q)/\ker \rho_{13} \cong G/\ker \rho_{13}|_G$.
  If we identify $G = \langle Q_1, Q_2 \rangle$ with $\Z^2$, then we can
  find the kernels of $\rho_7$ and of~$\rho_{13}$ on~$G$ explicitly, and
  we can also determine the relevant cosets explicitly. It can then be checked
  that the union of the first two cosets does not meet the union of the
  second two cosets. This implies that such a point $P$ cannot exist.
  Therefore, the only remaining possibilities are that
  $\rho_7(P) = \rho_7(\infty_\pm)$.
\end{proof}

\begin{Remark}
  The use of information at $p = 13$ to rule out residue classes at $p = 7$
  in the proof above is a very simple instance of a method known as the
  {\em Mordell-Weil sieve}. For a detailed description of this method,
  see~\cite{BSMWS}.
\end{Remark}

Now we need to find the space of holomorphic 1-forms on~$C_1$, defined over~$\Q_7$,
that annihilate the Mordell-Weil group under the integration pairing,
compare Sect.~\ref{SS:Chab}.
We follow the procedure described in~\cite{StollXdyn06}. We first find
two independent points in the intersection of $J_1(\Q)$ and the kernel
of reduction mod~7. In our case, we take $R_1 = 20 Q_1$ and
$R_2 = 5 Q_1 + 60 Q_2$. We represent these points in the form
$R_j = [D_j - 4 \infty_-]$ with effective divisors $D_1$, $D_2$ of
degree~4. The coefficients of the primitive polynomial in~$\Z[x]$ whose
roots are the $x$-coordinates of the points in the support of~$D_1$
have more than~100 digits and those of the corresponding polynomial
for~$D_2$ fill several pages, so we refrain from printing them here.
(This indicates that it is a good idea to work with a small prime!)
The points in the support of $D_1$ and~$D_2$ all reduce
to~$\infty_-$ modulo the prime above~7 in their fields of definition
(which are degree~4 number fields totally ramified at~7). Expressing
a basis of $\Omega^1_{C_1}(\Q_7)$ as power series in the uniformiser
$t = 1/x$ at~$P_0 = \infty_-$ times~$dt$, we compute the integrals numerically.
More precisely, the differentials
\[ \eta_0 = \frac{dx}{2y}, \quad \eta_1 = \frac{x\,dx}{2y}, \quad
   \eta_2 = \frac{x^2\,dx}{2y} \quad\text{and}\quad \eta_3 = \frac{x^3\,dx}{2y}
\]
form a basis of~$\Omega_{C_1}^1(\Q_7)$. We get
\[ \eta_j = t^{3-j} \Bigl(\frac{1}{2} - \frac{15}{2} t + 115 t^2 - 1980 t^3
                           + \frac{145385}{4} t^4 - \frac{2764899}{4} t^5 + \dots\Bigr)
                           \,dt
\]
as power series in the uniformiser. Using these power series up to a precision
of~$t^{20}$, we compute the following 7-adic approximations to the integrals.
\[ \Bigl(\int_0^{R_j} \eta_i\Bigr)_{0 \le i \le 3, 1 \le j \le 2}
    = \begin{pmatrix}
       -20 \cdot 7 + O(7^4)  & -155 \cdot 7 + O(7^4) \\
       -150 \cdot 7 + O(7^4) & -13 \cdot 7 + O(7^4) \\
       -130 \cdot 7 + O(7^4) & -83 \cdot 7 + O(7^4) \\
       -19 \cdot 7 + O(7^4)  & 163 \cdot 7 + O(7^4)
      \end{pmatrix}
\]
From this, it follows easily that the reductions mod~7 of the (suitably scaled)
differentials that kill~$J_1(\Q)$ fill the subspace of $\Omega^1_{C_1}(\F_7)$
spanned by
\[ \omega_1 = (1 + 3 x - 2 x^2) \frac{dx}{2 y} \quad\text{and}\quad
   \omega_2 = (1 - x^2 + x^3) \frac{dx}{2 y} \,.
\]
Since $\omega_2$ does not vanish at the points $\rho_7(\infty_\pm)$, this
implies that there can be at most one rational point~$P$ on~$C_1$ with
$\rho_7(P) = \rho_7(\infty_+)$ and at most one point~$P$ with
$\rho_7(P) = \rho_7(\infty_-)$ (see for example~\cite[Prop.~6.3]{StollChabauty}).

\begin{Proposition} \label{Prop1}
  The only rational points on~$C_1$ are $\infty_+$ and~$\infty_-$.
\end{Proposition}

\begin{proof}
  Let $P \in C_1(\Q)$. By Lemma~\ref{LemmaRed}, $\rho_7(P) = \rho_7(\infty_\pm)$.
  By the argument above, for each sign $s \in \{+,-\}$, we have
  $\#\{P \in C_1(\Q) : \rho_7(P) = \rho_7(\infty_s)\} \le 1$. These two
  facts together imply that $\#C_1(\Q) \le 2$. Since we know the two rational
  points $\infty_+$ and~$\infty_-$ on~$C_1$, there cannot be any further
  rational points.
\end{proof}

We can now prove Thm.~\ref{Thm}.

\begin{proof}[of Thm.~\ref{Thm}]
  The considerations
  in Sect.~\ref{Curves} imply that if $(a^2, b^2, c^2, d^5)$ is an
  arithmetic progression in coprime integers, then there are coprime $u$ and~$v$,
  related to $a,b,c,d$ by~\eqref{rel1},
  such that $(u/v, a/v^5)$ is a rational point on one of the curves~$C_j$
  with $-2 \le j \le 2$. By Prop.~\ref{Prop02}, there are no rational
  points on $C_0$ and~$C_2$ and therefore also not on the curve~$C_{-2}$, which
  is isomorphic to~$C_2$. By Prop.~\ref{Prop1}, the only rational points
  on~$C_1$ (and~$C_{-1}$) are the points at infinity. This translates into
  $a = \pm 1$, $u = \pm 1$, $v = 0$, and we have $j = \pm 1$. We deduce
  $a^2 = 1$, $b^2 = g_1(\pm 1, 0)^2 = 1$, whence also $c^2 = d^5 = 1$.
\end{proof}


\end{document}